\documentclass[12pt,letterpaper]{amsart}
\usepackage[utf8]{inputenc}
\usepackage{amsmath}
\usepackage{amsfonts}
\usepackage{amssymb}
\usepackage{amsthm}
\usepackage{mathrsfs}
\usepackage{tikz-cd}
\usepackage{subcaption}
\usepackage[style=alphabetic]{biblatex}
\usepackage[left=1in,right=1in,top=1in,bottom=1in]{geometry}
\author{Eric Zhu}
\email{eric\_zhu1@brown.edu}
\title{Brauer-Manin obstruction on generalized Kummer varieties}
\date{\today}
\newtheorem{theorem}{Theorem}[section]
\newtheorem{corollary}[theorem]{Corollary}
\newtheorem{proposition}[theorem]{Proposition}
\newtheorem{lemma}[theorem]{Lemma}
\theoremstyle{definition}
\newtheorem{example}[theorem]{Example}
\theoremstyle{remark}
\newtheorem{remark}[theorem]{Remark}

\newcommand{\bZ}{\mathbb{Z}}

\newcommand{\bQ}{\mathbb{Q}}

\newcommand{\et}{\mathrm{\grave{e}t}}
\begin{document}

\begin{abstract}
Given an abelian variety $A$ over a number field, we consider the generalized Kummer varieties of $A$ coming from quotients of $A$ by an automorphism of prime order $p > 2$. We prove that the Brauer-Manin obstruction on these generalized Kummer varieties only can come from the $p$-primary part of the Brauer group. This is applied to show that certain families of such varieties have no Brauer-Manin obstruction to the local-global principle.
\end{abstract}
\date{\today}
\maketitle
\section{Introduction}

To determine whether a "nice" variety $X$ over a number field $k$ has rational points, the local-global principle says to test whether it has $k_v$-points for all places $v$ of $k$. However, even if a variety has $k_v$-points for all $v$, it still may not have a rational point. Due to this, we often work with the Brauer-Manin obstruction, a refinement of the local-global principle. This arises from the Brauer-Manin pairing $X(\mathbb{A}_k) \times \mathrm{Br}(X) \to \mathbb{Q}/ \mathbb{Z}$ coming from global class field theory. Here, the Brauer group $\mathrm{Br}(X)$ is defined as $\mathrm{H}^2_\et(X, \mathbb{G}_m)$. For any subgroup $B \subseteq \mathrm{Br}(X)$, we denote $X(\mathbb{A}_k)^B$ to be the left kernel of $B$ with respect to the Brauer-Manin pairing and denote $X(\mathbb{A}_k)^\mathrm{Br} = X(\mathbb{A}_k)^{\mathrm{Br}(X)}$ to be the Brauer-Manin set. This set satisfies the property that $X(k) \subseteq X(\mathbb{A}_k)^\mathrm{Br} \subseteq X(\mathbb{A}_k)$ and so can be used to detect an obstruction to having rational points, even if $X(\mathbb{A}_k)$ is nonempty.

A fruitful line of study in the Brauer-Manin obstruction is to narrow down the Brauer group to a smaller subgroup that captures the entire obstruction. For example, in \cite{CreutzVirayDegree}, they show that for torsors of abelian varieties over number fields, the Brauer-Manin obstruction is completely determined by the $d$-primary part of the Brauer group where $d$ is the period of the torsor. For curves, this was studied in \cite{CreutzVirayVolochCurves}. Similarly, in \cite{SkorobogatovZarhinKummer}, they prove that for Kummer varieties, the 2-primary subgroup of the Brauer group captures the Brauer-Manin obstruction. These results are important for both computational and theoretical reasons.	In general, the Brauer group can be quite large and so restricting the problem to only looking at a much smaller subgroup can be useful.

To this end, we prove the following:
\begin{theorem}
Let $A$ be an abelian variety of dimension $g \geq 2$ over a number field $k$ with an automorphism $\zeta$ of prime order $p$ defined over $k$. Further assume that $\zeta$ only has finitely many fixed points in $A(\overline{k})$. Let $X$ be a generalized Kummer variety associated to $A$ and $\zeta$, and let $\mathrm{Br}(X)[p^\perp]$ be the prime-to-$p$ torsion subgroup of $\mathrm{Br}(X)$. If $B \subseteq \mathrm{Br}(X)$ satisfies $X(\mathbb{A}_k)^B \neq \emptyset$, then $X(\mathbb{A}_k)^{B + \mathrm{Br}(X)[p^\perp]} \neq \emptyset$.
\end{theorem}

These fixed points of $\zeta$ form a finite subgroup of $A$ that we denote $T$ or $A[1 - \zeta]$. These fixed points mean that the quotient $A/\langle \zeta \rangle$ is singular and the generalized Kummer varieties $X$ we study in this paper are certain resolutions of this quotient. In \cite{SkorobogatovZarhinKummer} and \cite{CreutzVirayDegree}, they prove this statement for Kummer varieties when $p = 2$. The work of this paper extends this result to higher order automorphisms. In other words, in the language of \cite{viray2023rationalpointsvarietiesbrauermanin}, we prove that the $p$-primary part of the Brauer group completely captures the Brauer-Manin obstruction for generalized Kummer varieties.

Since $A$ is defined over $k$, it has a rational point by definition, and therefore so do its quotients and resolutions. So the discussion of the Brauer-Manin obstruction for these varieties is trivial. Therefore, we will work with certain torsors of $A$. We note that since $T \subseteq A$, we have a natural map $\mathrm{H}^1(k, T) \to \mathrm{H}^1(k, A)$. We will work with torsors $Y$ of $A$ coming from the image of this map. Note that these are exactly the torsors such that $\zeta$ gives an automorphism on Y defined over $k$, which we will also denote $\zeta$. These torsors are a subset of what are known as the $p$-coverings of the abelian variety, since the map $\mathrm{H}^1(k,T) \to \mathrm{H}^1(k,A)$ factors through the map $\mathrm{H}^1(k, A[p]) \to \mathrm{H}^1(k, A)$.

In this paper, for a variety $X$ over a number field $k$, we let $\overline{X}$ denote the base change to the algebraic closure $\overline{k}$. In the study of the Brauer group, there are two important subgroups. Let $\mathrm{Br}_0(X) \subseteq \mathrm{Br}_1(X) \subseteq \mathrm{Br}(X)$, where $\mathrm{Br}_0(X)$ is the image of $\mathrm{Br}(k) \to \mathrm{Br}(X)$ and $\mathrm{Br}_1(X)$ is the kernel of $\mathrm{Br}(X) \to \mathrm{Br}(\overline{X})$. Class field theory implies that the elements in $\mathrm{Br}_0(X)$ pair trivially with the elements in $X(\mathbb{A}_k)$ so for the sake of the Brauer-Manin obstruction, it suffices to study $\mathrm{Br}(X)/\mathrm{Br}_0(X)$. 

\section{Construction of Generalized Kummer Varieties}
In this section, we discuss how to construct generalized Kummer varieties. As mentioned in the introduction, let $A$ be an abelian variety of dimension $g \geq 2$ over a number field $k$ with an automorphism $\zeta$ of $A$ of prime order $p > 2$ defined over $k$. We take a torsor $Y$ of $A$ coming from an element of $\mathrm{H}^1(k, A)$ living in the image of $\mathrm{H}^1(k,T) \to \mathrm{H}^1(k,A)$. Due to this choice, there is an automorphism on $Y$ defined over $k$ that is compatible with the automorphism of $\zeta$ on $A$ and we denote the automorphism as $\zeta$ on $Y$ as well.

To construct a resolution of singularities for $Y/\langle \zeta \rangle$, it suffices to work \'etale locally. The action of $\zeta$ on $Y$ around a fixed point looks like the action on $\mathbb{A}^n$ of some order $p$ matrix in $\mathrm{GL}_n$, which we call $M$. We can think of this matrix as acting on the ring $k[x_1, \dots, x_n]$. Then, the spectrum of the ring of invariants will be exactly the quotient variety $\mathbb{A}^n/M$. This quotient is a toric variety with a cyclic quotient singularity, and therefore we can resolve this using a toric resolution corresponding to a subdivision of the fan. Furthermore, we can choose a simple normal crossings resolution, and we fix such a resolution to use. We can represent this resolution as a blowup by some ideal $I$ of the ring of invariants. Then, we can consider the ideal $J$ generated by $I$ inside of $k[x_1, \dots, x_n]$. Blowing up $\mathbb{A}^n$ at $J$ satisfies the property that $(\mathrm{Bl}_J\mathbb{A}^n)/M \cong \mathrm{Bl}_I(\mathbb{A}^n/M)$. This is because the exceptional divisor in $\mathrm{Bl}_J\mathbb{A}^n$ is given by $J \oplus J^2 \oplus \cdots$ since the blow up is the Proj of the Rees algebra. Since $J$ is invariant under $M$, the exceptional divisor is as well. However $\mathrm{Bl}_J \mathbb{A}^n$ may not be smooth. This blowup corresponds to lifting the subdivision from the fan of $\mathbb{A}^n/M$ to the cone for $\mathbb{A}^n$. Therefore, we can subdivide this subdivision further to get a resolution of $\mathrm{Bl}_J \mathbb{A}^n$.

To transfer this back to $Y$, we note that around any fixed point $x \in T$, there is an open affine subvariety $x \in U \subset Y$. Furthermore, by replacing $U$ with $U \cap \zeta U \cap \cdots \cap \zeta^{p-1}U$, we can assume that $\zeta$ acts on $U$. Then, the Luna map gives us a $\zeta$-equivariant \'etale morphism from $U \to \mathbb{A}^n$. Let $Y'$ be the variety obtained by gluing the schemes $\mathrm{Bl}_{\pi^{-1}(J)} U = \mathrm{Bl}_J \mathbb{A}^n \times_{\mathbb{A}^n} U$ for each fixed point of $\zeta$ along with the open set that is the complement of all the fixed points. The \'etale local computation is enough to guarantee that the exceptional divisors on the blowup of $U$ are the same as the blowup of $\mathbb{A}^n$ by the same argument as in Proposition 3.2 in \cite{navone2025transcendentalbrauergroupscubic}.

Then, $\zeta$ lifts to an automorphism of $Y'$ by the universal property of blow ups. We define $X = Y' / \langle \zeta \rangle$ to be a generalized Kummer variety.  By the construction of the resolution, $X$ is smooth. Let $\sigma : Y' \to Y$ be the blow up map and $\pi : Y' \to X$ be the quotient map. The variety $Y'$ may not be smooth, but will be normal. Furthermore, we let $Y''$ be a resolution of $Y'$.

\begin{example} 
\label{ex:E}
If $E$ is an elliptic curve, then we can consider the automorphism on $E \times E$ that maps $(x,y)$ to $(y, -x-y)$. This is an order 3 automorphism on $E \times E$ and so we can construct the generalized Kummer variety $X$ that was studied in \cite{van2007cubic}. Here, the minimal resolutions gives $\mathbb{P}^1 \vee \mathbb{P}^1$ as exceptional divisors above each of the 3 fixed points of the action. In this example, the toric geometry of the resolutions is as follows. The variety $\mathbb{A}^2$ is given by the cone generated by $(1,0)$ and $(0,1)$ in the lattice $\mathbb{Z}^2$. 

\begin{figure}[h]
\centering
\begin{subfigure}[b]{0.29\textwidth}
\begin{tikzpicture}
\draw[->] (-1,0)--(3,0) node[right]{$x$};
\draw[->] (0,-1)--(0,3) node[above]{$y$};
\draw[-] (0,0)--(1/3,2/3) node[right, yshift=2pt]{\tiny ${\scriptstyle (\frac13,\frac23)}$};
\draw[-] (0,0)--(2/3,1/3) node[right, yshift=-2pt]{\tiny ${\scriptstyle (\frac23,\frac13)}$};
\draw[thick,->] (0,0)--(3,3/2);
\draw[thick,->] (0,0)--(3/2,3);
\fill (3,3) circle(1pt);
\foreach \m in {0,...,2}
  \foreach \n in {1,...,3}{
    \pgfmathsetmacro{\a}{\m}
    \pgfmathsetmacro{\b}{\n}
    \pgfmathsetmacro{\c}{\m + 1/3}
    \pgfmathsetmacro{\d}{\n - 1/3}
    \pgfmathsetmacro{\e}{\m + 2/3}
    \pgfmathsetmacro{\f}{\n - 2/3}
    \pgfmathsetmacro{\g}{\m + 1}
    \pgfmathsetmacro{\h}{\n - 1}
    \fill (\a,\b) circle(1pt);
    \fill (\c,\d) circle(1pt);
    \fill (\e,\f) circle(1pt);
    \fill (\g,\h) circle(1pt);
}
\end{tikzpicture}
\label{Resolution1}
\caption*{\text{Resolution of Quotient of $\mathbb{A}^2$}}
\end{subfigure}
\hfill
\begin{subfigure}[b]{0.28\textwidth}
\begin{tikzpicture}
\draw[->] (-1,0)--(3,0) node[right]{$x$};
\draw[->] (0,-1)--(0,3) node[above]{$y$};
\draw[-] (0,0)--(2,1) node[right, yshift=-2pt]{${\scriptstyle (2,1)}$};
\draw[-] (0,0)--(1,2) node[right]{$\scriptstyle (1,2)$};
\foreach \x in {0,...,3}
  \foreach \y in {0,...,3}
    \fill (\x,\y) circle(1pt);
\draw[thick,->] (0,0)--(3,3/2);
\draw[thick,->] (0,0)--(3/2,3);
\end{tikzpicture}
\label{Resolution2}
\caption*{\text{Lift of Resolution to $\mathbb{A}^2$}}
\end{subfigure}
\hfill
\begin{subfigure}[b]{0.36\textwidth}
\begin{tikzpicture}
\draw[->] (-1,0)--(3,0) node[right]{$x$};
\draw[->] (0,-1)--(0,3) node[above]{$y$};
\draw[-] (0,0)--(2,1) node[right, yshift=-4pt, xshift=-3pt]{${\scriptstyle(2,1)}$};
\draw[thick,->] (0,0)--(3,3/2);
\draw[-] (0,0)--(1,1) node[right]{${\scriptstyle(1,1)}$};
\draw[thick,->] (0,0)--(3,3);
\draw[-] (0,0)--(1,2) node[right, xshift=-1pt]{${\scriptstyle (1,2)}$};
\draw[thick,->] (0,0)--(3/2,3);
\foreach \x in {0,...,3}
  \foreach \y in {0,...,3}
    \fill (\x,\y) circle(1pt);
\end{tikzpicture}
\label{Resolution3}
\caption*{\text{Resolution of the Blowup of $\mathbb{A}^2$}}
\end{subfigure}
\end{figure}

Then the variety given by quotienting $\mathbb{A}^2$ by the action $(x,y) \to (y, -x -y)$ is represented by the same cone as before, but now with the lattice $\mathbb{Z}^2 + \mathbb{Z}\left(\frac{1}{3}, \frac{2}{3} \right)$. This cone is not a smooth cone in this lattice, but we can subdivide the cone with the rays going through $\left( \frac{1}{3}, \frac{2}{3} \right)$ and $\left(\frac{2}{3}, \frac{1}{3} \right)$ giving us three smooth cones. Lifting this resolution back to the lattice $\mathbb{Z}^2$ gives us the rays through $(1,2)$ and $(2,1)$, but the cone defined by these two rays is not smooth. Therefore, we can subdivide this further with the ray through $(1,1)$ to get a smooth subdivision of the original cone.
\end{example}

Now, let us prove some facts about the geometry of such resolutions. Let $E = \sigma^{-1}(T)$ be the exceptional divisor in $Y'$, and let $D = \pi(E) \subset X$ be the image in $X$.

For any generalized Kummer variety $X$, each irreducible component of $D$ is simply connected, since they are complete toric varieties and therefore are simply connected due to Section 3.2 in \cite{fulton1993introduction}.

\begin{lemma}
The irreducible exceptional divisors of $X \to Y/\langle\zeta\rangle$ are linearly independent in $\mathrm{Pic}(X)$.
\end{lemma}
More generally, the following is true: 
\begin{lemma}
For any proper, birational morphism $f: \tilde{S} \to S$ of integral normal varieties, the exceptional divisors $E_1, \dots, E_n$ are linearly independent in $\mathrm{Cl}(\tilde{S})$.
\end{lemma}
\begin{proof}
Assume that there is some $g \in k(\tilde{S})$ such that $\mathrm{div}_{\tilde{S}}(g) = \sum a_i E_i$. Then, the fact that $f$ is a birational morphism tells us that $k(\tilde{S}) \cong k(S)$, so we can also think of $g$ as a rational function on $S$. Now, on $S$, $\mathrm{div}_S(g) = 0$, which implies that $g$ is a unit in $\mathcal{O}_S(S)$ that is integrally closed, since $S$ is an integral normal variety. This implies that $g$ must also be a unit in $\mathcal{O}_{\tilde{S}}(\tilde{S})$, so that $\mathrm{div}_{\tilde{S}}(g) = 0$.
\end{proof}
\section{Proving that $\mathrm{Br}(X)/\mathrm{Br}(X_0)$ is finite}

Let $Y_0 = Y \backslash T$ and $X_0 = \pi(\sigma^{-1}(Y_0))$, which is the smooth locus of $Y/\langle \zeta \rangle$.

\begin{proposition}
The abelian group $\mathrm{Pic}(\overline{X})$ is torsion-free.
\end{proposition}
\begin{proof}
We take the commutative group scheme $\mathcal{G} = T \times \bZ /p$, which acts on $A$; where the elements of $T$ act as translation and the generator of $\bZ/p$ acts as $\zeta$. We let $A_1 = A\backslash A[(1 - \zeta)^2]$ and note that $\mathcal{G}$ acts freely on $A_1$. Furthermore, the quotient $A_1/\mathcal{G}$ is exactly $X_0$. This is because $$A_1/\mathcal{G} = (A\backslash A[(1 - \zeta)^2])/(A[1 - \zeta] \times \bZ/p) = (A\backslash A[1 - \zeta])/\bZ/p = X_0.$$ Using this, we get an injection $\hat{\mathcal{G}} \to \mathrm{Pic}(\overline{X}_0)$ using \cite{SkorobogatovTorsors}, where $\hat{\mathcal{G}}$ is the Cartier dual of $\mathcal{G}$. Now, we look to argue that $\mathrm{Pic}(\overline{X}_0)_{\mathrm{tors}}$ has the same order as $\hat{\mathcal{G}}$, so that we can see that $\mathrm{Pic}(\overline{X}_0)_{\mathrm{tors}} \cong \hat{\mathcal{G}}$ as abelian groups.

We first note that $\overline{A}_0 \to \overline{X}_0$ is a $\bZ/p$-torsor, since we have removed all the fixed points of $\zeta$ from $A$. Therefore, we can apply Hochschild-Serre spectral sequence $$\mathrm{H}^p(\bZ/p, \mathrm{H}^q_{\et}(\overline{A}_0, \mathbb{G}_m)) \Rightarrow \mathrm{H}^{p+q}(\overline{X}_0, \mathbb{G}_m),$$ which gives us the exact sequence \begin{equation}0 \to \bZ/p \to \mathrm{Pic}(\overline{X}_0) \to \mathrm{Pic}(\overline{A}_0)^{\zeta} \to 0, \label{eq:1}
\end{equation}
 where the last term is the group of elements of $\mathrm{Pic}(\overline{A}_0)$ fixed by the induced action of $\zeta$. The first term is $\mathrm{H}^1(\bZ/p, \overline{k}^*) = \bZ/p$, and the last zero is $\mathrm{H}^2(\bZ/p, \overline{k}^*) = \overline{k}^*/\overline{k}^{*p} = 0$.

The other piece we will need is the short exact sequence $$ 0 \to A^t(\overline{k}) \to \mathrm{Pic}(\overline{A}) \to\mathrm{NS}(\overline{A}) \to 0. $$The automorphism $\zeta$ on $A$ induces an action of $\zeta$ on the above abelian groups as well. So taking the long exact sequence associated to group cohomology gives us \begin{equation} \label{eq:2} 0 \to A^t[1 - \zeta] \to \mathrm{Pic}(\overline{A})^{\zeta} \to \mathrm{NS}(\overline{A})^\zeta \to 0\end{equation} where the zero at the end comes from $\mathrm{H}^1(\langle \zeta \rangle, A^t(\overline{k}))$, which we can calculate using Tate cohomology of cyclic groups as $$\ker(\mathrm{Norm})/A^t(\overline{k})I_{\langle \zeta \rangle} = A^t(\overline{k})/(1 - \zeta)A^t(\overline{k}) = 0$$

Equations \ref{eq:1} and \ref{eq:2} come together to give us the commutative diagram $$\begin{tikzcd}
            &                                  & 0                                                                & 0                                                    &   \\
            &                                  & \mathrm{NS}(\overline{A})^\zeta \arrow[u] \arrow[equal]{r}                              & \mathrm{NS}(\overline{A})^\zeta \arrow[u]                  &   \\
0 \arrow[r] & \mathbb{Z}/p \arrow[r]           & \mathrm{Pic}(\overline{X}_0) \arrow[r] \arrow[u]                 & \mathrm{Pic}(\overline{A})^\zeta \arrow[r] \arrow[u] & 0 \\
0 \arrow[r] & \mathbb{Z}/p \arrow[equal]{u} \arrow[r] & \mathrm{Pic}(\overline{X}_0)_{\mathrm{tors}} \arrow[r] \arrow[u] & T^t \arrow[r] \arrow[u]                              & 0 \\
            &                                  & 0 \arrow[u]                                                      & 0 \arrow[u]                                          &  
\end{tikzcd} $$ where the rows and columns are exact. Therefore, looking at the bottom row, we see that the order of $\mathrm{Pic}(\overline{X}_0)_{\mathrm{tors}}$ is exactly $T^t \times \bZ/p$, so the earlier injection tells us that $\mathrm{Pic}(\overline{X}_0)_{\mathrm{tors}} \cong T^t \times \bZ/p$.

The kernel of the map $\mathrm{Pic}(\overline{X}) \to \mathrm{Pic}(\overline{X}_0)$ is generated by the irreducible components of the exceptional divisor, which we showed is torsion free, so this map is an injective map on torsion subgroups. Therefore, to show that $\mathrm{Pic}(\overline{X})$ is torsion-free, we only need to show that it has no $p$-torsion. This $p$-torsion corresponds to connected unramified $p$-coverings of $\overline{X}$. A normal finite surjective covering of $\overline{X}$ is determined by its restriction to $\overline{X}_0$. So we look to show that any connected unramified $p$-covering of $\overline{X}_0$ extends to a ramified covering of $\overline{X}$. But since $\mathrm{Pic}(\overline{X}_0)_\mathrm{tors} \cong T \times \bZ/p$, all connected unramified $p$-coverings of $\overline{X}_0$ must come from an index $p$ subgroup $\mathcal{H}$ of $\mathcal{G}$, giving us the covering $\overline{A}_1/\mathcal{H}$ of $\overline{X}_0$.

As abstract groups, $T$ is a subgroup of $A[p] \cong (\bZ/p)^{2g}$, so it has the form $(\bZ/p)^b$ for some $b$. Therefore, by Goursat's lemma, we can separate index $p$ subgroups of $(\bZ/p)^b \times \bZ/p$ into three cases. In the first case, we take the subgroup $\mathcal{H} = T$. Then, $\overline{A}_1 / T$ is $\overline{A} \backslash T = \overline{A}_0$. Therefore, the unramified covering $\overline{A}_0 \to \overline{X}_0$ extends to the covering $\overline{A}' \to \overline{X}$ which is ramified at the exceptional divisors. 

The other two cases only occur if $T$ is non-trivial, so we assume that. These two cases occur when the subgroup surjects onto the second factor, and we have two cases depending on whether this surjection splits. If this is split, then we take a homomorphism $\phi : T \to \bZ/p$ and take the index $p$ subgroup $\mathcal{H}$ of $T \times \bZ/p$ to be the kernel of $(x,y) \mapsto \phi(x)$. In this case, take $A_\phi = \overline{A}/\ker(\phi)$ and remove the image of $A[(1 -\zeta)^2]$ in $A_\phi$. Then we have that $\overline{A}_1/ \mathcal{H} = A_\phi/\langle \zeta \rangle$, so the unramified cover $\overline{A}_1 /\mathcal{H} \to \overline{X}_0$ extends to the covering of $\overline{A}/\mathcal{H} \to \overline{X}$ ramified at $\sigma^{-1}(T \backslash \ker{\phi})$.

For the third case, we again take a homomorphism $\phi : T \to \bZ/p$, but this time take the index $p$ subgroup of $T \times \bZ/p$ to be the kernel of $(x,y) \mapsto \phi(x) - y$. Again, we let $A_\phi = \overline{A}/\ker{\phi}$ and remove $A[(1 - \zeta)^2]$. Here we need to choose an $a \in A[1 - \zeta]$ such that $\phi(a) \neq 0$. By abuse of notation, let $\phi$ also denote the map from $A \to A_\phi$, and note that $\zeta$ descends to an automorphism of $A_\phi$. Then, $\overline{A}_1/\mathcal{H}$ is the quotient of $A_\phi$ by $x \mapsto \phi(a) + \zeta x$ and compositions of this map, and so the unramified covering $\mathcal{A}_1/\mathcal{H} \to \overline{X}_0$ is the restriction of the covering of $\overline{X}$ ramified in $\sigma^{-1}(\ker(\phi))$.
\end{proof}

\begin{proposition}
\label{prop:finiteKernel}
The composition of maps $$\mathrm{Br}(\overline{X}) \rightarrow \mathrm{Br}(\overline{Y''}) \xrightarrow{\sim} \mathrm{Br}(\overline{Y})
$$ has finite kernel.
\end{proposition}

\begin{proof}
Note that $\mathrm{Br}(\overline{Y''}) \cong \mathrm{Br}(\overline{Y})$ due to the birational invariance of the Brauer group for smooth projective varieties. We first see that Grothendieck's exact sequence gives us $$0 \rightarrow \mathrm{Br}(\overline{X}) \rightarrow \mathrm{Br}(\overline{X}_0) \rightarrow \bigoplus \mathrm{H}^1(D_i, \bQ/ \bZ),$$ where the terms in the direct sum run over all irreducible divisors in $X \backslash X_0$. These last terms are zero since each $D_i$ is simply connected, so $\mathrm{Br}(\overline{X}) \cong \mathrm{Br}(\overline{X}_0)$, and since $T$ has codimension at least 2 in $Y$, we also get an isomorphism $\mathrm{Br}(\overline{Y}) \cong \mathrm{Br}(\overline{Y}_0)$. Therefore it suffices to analyze the map $\mathrm{Br}(\overline{X}_0) \to \mathrm{Br}(\overline{Y}_0)$, which we can do using the Hochschild-Serre spectral sequence. In particular, we get an exact sequence $$0 \to \ker\left(\mathrm{Br}(\overline{X}_0) \to \mathrm{Br}(\overline{A}_0)^\zeta \right) \to \mathrm{H}^1(\langle \zeta \rangle, \mathrm{Pic}(\overline{A}_0)).$$ Since $\mathrm{Pic}(\overline{A}_0) \cong \mathrm{Pic}(\overline{A})$, we can analyze $\mathrm{H}^1(\langle \zeta \rangle, \mathrm{Pic}(\overline{A}))$ instead. We again use the short exact sequence $$0 \to A^t(\overline{k}) \to \mathrm{Pic}(\overline{A}) \to \mathrm{NS}(\overline{A}) \to 0,$$ and again the action of $\zeta$ gives us an exact sequence $$\mathrm{H}^1(\langle \zeta \rangle, A^t(\overline{k})) \to \mathrm{H}^1(\langle \zeta \rangle, \mathrm{Pic}(\overline{A})) \to \mathrm{H}^1(\langle \zeta \rangle, \mathrm{NS}(\overline{A})).$$ But since we saw earlier that $\mathrm{H}^1(\langle \zeta \rangle, A^t(\overline{k}))$ is trivial, it suffices to show that $\mathrm{H}^1(\langle \zeta \rangle, \mathrm{NS}(\overline{A}))$ is finite. However, this is just the group cohomology of a finite group acting on a finitely generated group and so must be finite. This tells us that $\ker\left(\mathrm{Br}(\overline{X}_0) \to \mathrm{Br}(\overline{A}_0)^\zeta \right)$ is finite, so the map $\mathrm{Br}(\overline{X}) \to \mathrm{Br}(\overline{A})$ is a Galois equivariant map with finite kernel.
\end{proof}

\begin{corollary}
The group $\mathrm{Br}(X)/\mathrm{Br}_0(X)$ is finite.
\end{corollary}
\begin{proof}
\label{cor:finiteBrauer}
We note that the spectral sequence $\mathrm{H}^p(k, \mathrm{H}^q_\et(\overline{X}, \mathbb{G}_m)) \Rightarrow \mathrm{H}^{p + q}_\et (X, \mathbb{G}_m)$ tells us that $\mathrm{Br}_1(X)/\mathrm{Br}_0(X) \cong \mathrm{H}^1(k, \mathrm{Pic}(\overline{X}))$, which is finite since $\mathrm{Pic}(\overline{X})$ is torsion free and thus finitely generated. Therefore, to show the result, we just need to show that $\mathrm{Br}(X)/\mathrm{Br}_1(X)$ is finite. The higher degree terms of the spectral sequence tell us that $\mathrm{Br}(X)/\mathrm{Br}_1(X)$ injects into $\mathrm{Br}(\overline{X})^{G_k}$. Proposition \ref{prop:finiteKernel} tells us that the map $\mathrm{Br}(\overline{X}) \to \mathrm{Br}(\overline{A})$ is Galois equivariant and has finite kernel. So, the map $\mathrm{Br}(\overline{X})^{G_k} \to \mathrm{Br}(\overline{A})^{G_k}$ also has finite kernel and the codomain is finite by \cite{SkorobogatovZarhinFinite}, which tells us that $\mathrm{Br}(\overline{X})^{G_k}$ is finite.
\end{proof}

\section{Brauer-Manin Obstruction}

\begin{proposition}
\label{prop:decomposition}
For any integer $n > 1$ such that $n \equiv 1 \pmod{p}$, we have a decomposition $$\mathrm{H}^2_\et(Y, \mu_n) = \mathrm{H}^2_\et(k, \mu_n) \oplus \mathrm{H}^1(k, \mathrm{H}^1_\et(\overline{Y}, \mu_n)) \oplus \mathrm{H}^2(\overline{Y}, \mu_n)^{G_k}$$ that respects the action of $\zeta$. Note that $\zeta$ acts trivially on $\mathrm{H}_\et^2(k, \mu_n)$ and non-trivially on the other two summands.
\end{proposition}
\begin{proof}
We have maps $\alpha: \mathrm{H}^2_\et(k, \mu_n) \to \mathrm{H}^2_\et(Y, \mu_n)$ and $\beta: \mathrm{H}^2_\et(Y, \mu_n) \to \mathrm{H}^2_\et(\overline{Y}, \mu_n)^{G_k}$ coming from the maps $k \to Y$ and $\overline{Y} \to Y$. To prove the splitting of $\mathrm{H}^2_\et(Y,\mu_n)$, we construct a retraction of $\alpha$ and a section of $\beta$. 

To get a retraction of $\alpha$, we use the fact that $T$ is zero-dimensional. This means we can write it as $\bigcup_{i = 1}^r \mathrm{Spec}(k_i)$ for some finite extensions $k_i/k$. Now, we have the restriction map $$\mathrm{H}^2_\et(Y, \mu_n) \to \mathrm{H}^2_\et(T, \mu_n) = \bigoplus_{i = 1}^t \mathrm{H}^2(k_i, \mu_n).$$ For each $i$, we have the restriction map $\mathrm{H}^2(k, \mu_n) \to \mathrm{H}^2(k_i, \mu_n)$ and the corestriction map $\mathrm{H}^2(k_i, \mu_n) \to \mathrm{H}^2(k, \mu_n)$ whose composition is multiplication by $[k_i:k]$. We can take the sum of these corestriction maps to get a map $\mathrm{H}^2_\et(T, \mu_n) \to \mathrm{H}^2(k, \mu_n)$. The composition of the map $\mathrm{H}^2(k, \mu_n) \to \mathrm{H}^2_\et(T, \mu_n)$ with this corestriction map is multiplication by $t$ since the sum of the degrees $[k_i:k]$ is exactly the number of points in $T(\overline{k})$. Since $T$ is contained in $A[p]$, its order must be a power of $p$. This means we can find an integer $m$ such that $tm \equiv 1 \pmod n$, since $n$ is assumed to be coprime to $p$. Therefore, we can look at $$\mathrm{H}^2_\et(Y, \mu_n) \to \mathrm{H}^2_\et(T, \mu_n) \to \mathrm{H}^2(k, \mu_n) \to \mathrm{H}^2(k, \mu_n),$$ where the last map is multiplication by $m$, and note that the composition of all these maps is a retraction of $\alpha$, since the elements in $\mathrm{H}^2(k, \mu_n)$ have order dividing $n$.

To construct a section of $\beta$, first we note that since $Y$ is in the image of $\mathrm{H}^1(k,T) \to \mathrm{H}^1(k,A)$, it lies inside $\mathrm{H}^1(k,A)[1-\zeta]$ using the cohomology of the short exact sequence $0 \to T \to A \to A \to 0 $. Then, $Y$ is also in $\mathrm{H}^1(k,A)[p]$ and since $n$ is $1 \pmod{p}$, we get a morphism $[n]: Y \to Y$ defined over $k$ compatiable with the multiplication-by-$n$ map on $A$. Now, we define the torsor $\mathcal{T}_n$ coming from the covering $[n] : Y \to Y$. This torsor defines a class in $\mathrm{H}^1_\et(A, A[n])$. Taking wedge products, we get a class $\wedge^2 \mathcal{T}_n \in \mathrm{H}^2_\et(Y, \wedge^2 A[n])$. Using the fact that $\mathrm{H}^2_\et(\overline{Y}, \mu_n) = \mathrm{H}^2_\et(\overline{A}, \mu_n) = \mathrm{Hom}(\wedge^2 A[n], \mu_n)$, we get a pairing $$\mathrm{H}^2_\et(Y, \wedge^2 A[n]) \times \mathrm{H}^2_\et(\overline{Y}, \mu_n)^{G_k} \rightarrow \mathrm{H}^2_\et(Y, \mu_n).$$ Therefore we can define a map $s: \mathrm{H}^2_\et(\overline{Y}, \mu_n)^{G_k} \to \mathrm{H}^2_\et(Y, \mu_n)$ given by pairing with $\wedge^2 \mathcal{T}_n$. This map is proved to be a section in \cite{SkorobogatovZarhinKummerSurface}.

This retraction and section show that $\mathrm{H}^2(k, \mu_n)$ and $\mathrm{H}^2_\et(\overline{A}, \mu_n)^{G_k}$ are both direct summands of $\mathrm{H}^2_\et(A,\mu_n)$. Therefore, to finish the result, it is enough to show that $\ker(\beta)/\mathrm{Im}(\alpha) = \mathrm{H}^1(k, \mathrm{H}^1_\et(\overline{A},\mu_n))$ The short exact sequence coming from the spectral sequence $$\mathrm{H}^p(k, \mathrm{H}^q_\et(\overline{A}, \mu_n)) \Rightarrow \mathrm{H}^{p + q}_\et(A, \mu_n)$$ tells us that $$0 \to \ker(\beta)/\mathrm{Im}(\alpha) \to \mathrm{H}^1(k, \mathrm{H}^1_\et(\overline{A}, \mu_n)) \to 0$$ which is exactly what we want.
\end{proof}

Here, we note that we can take cyclic degree $p$ twists of $Y$. If $F$ is a cyclic degree $p$ extension of $k$, then we can form a twist of $Y$ by $F$ by taking a 1-cocycle $\xi \in \mathrm{H}^1(\mathrm{Gal}(F/k), \mathrm{Aut}(Y))$ such that $\xi$ maps into the subgroup of $\mathrm{Aut}(Y)$ generated by $\xi$. We denote this twist by $Y_\xi$. Then, we note that $Y_\xi$ also has an automorphism of order $p$ which we call $\zeta_\xi$. Furthermore, note that $Y[1 - \zeta] \cong Y_\xi[1 - \zeta_\xi]$ as $G_k$-modules due to the construction of the twist. We can apply the same resolution process as before to $Y_\xi$ and get a blow up $Y_\xi'$. Then, we can take the quotient $Y'_\xi/ \langle \zeta_\xi \rangle$ and get the same generalized Kummer variety $X$. We let $\sigma_\xi$ be the map $Y'_\xi \to Y_\xi$ and $\pi_\xi$ be the map $Y'_\xi \to X$. We let $Y''_\xi$ be a resolution of $Y'_\xi$. Let $Y_{\xi 0}$ be $Y_\xi \backslash Y_\xi[1- \zeta_\xi]$ and note that we can consider $Y_{\xi 0}$ as a subvariety of $Y'_\xi$ and $Y''_\xi$ as well.

\begin{corollary}
\label{cor:image}
For any integer $n > 1$ such that $n \equiv 1 \pmod{p}$ and for any $x \in \mathrm{H}^2_\et(X, \mu_n)$, there exists an $a \in \mathrm{H}^2(k, \mu_n)$ such that for any cyclic extension $F$ of $k$ of degree $p$ and twist $\xi$ of $Y$ by $F$, we have $\sigma_{\xi*} \pi^*_\xi(x)$ is the image of $a$ in $\mathrm{H}^2_\et(Y_\xi, \mu_n)$ from the map $\mathrm{H}^2(k, \mu_n) \to \mathrm{H}^2_\et(Y_\xi, \mu_n)$.
\end{corollary}
\begin{proof}
We note that $\sigma_{\xi*} \pi^*_\xi(x)$ is $\zeta_\xi$-invariant since the map $\sigma_{\xi*}$ is $\zeta_\xi$-equivariant and $\pi^*_\xi(x)$ is $\zeta_\xi$-invariant. So Proposition \ref{prop:decomposition} tells us that it must lie in the image of $\mathrm{H}^2(k, \mu_n) \to \mathrm{H}^2_\et(Y_\xi, \mu_n)$. We just need to prove that this $a \in \mathrm{H}^2(k, \mu_n)$ is independent of the twist $\xi$.

First, we let $D = X \backslash X_0$ be the exceptional divisor in $X$. Similarly, we let $D'' = Y''_\xi \backslash Y_{\xi 0}$ be the exceptional divisor in $Y''_\xi$. Note that $D$ and $D''$ do not depend on the twist $\xi$ since the blow-up locus stays the same as Galois modules under the twist. Therefore, we get a well-defined map $D'' \to D$ which does not depend on $\xi$. Since the irreducible components of $D''$ are toric varieties defined over $k$, we can choose a $k$-rational point in each to get a section $\rho : T \to D''$. Putting this all together with the natural restriction maps give us

\begin{equation}
\label{eq:diagram}
\begin{tikzcd}
{\mathrm{H}^2_\et(X, \mu_n)} \arrow[d, "\tau_1"] \arrow[r, "\pi_\xi^*"] \arrow[r] & {\mathrm{H}^2_\et(Y''_\xi, \mu_n)} \arrow[r, "\sigma_{\xi*}"] \arrow[d, "\tau_2"] & {\mathrm{H}^2_\et(Y_\xi, \mu_n)} \arrow[d, "\tau_3"] \arrow[l, "\sigma^*_\xi", shift left] \\
{\mathrm{H}^2_\et(D,\mu_n)} \arrow[r, "\pi^*|_{D''}"]                           & {\mathrm{H}^2_\et(D'' ,\mu_n)} \arrow[r, "\rho^*", shift left]               & {\mathrm{H}^2_\et(T, \mu_n)} \arrow[l]                                                 
\end{tikzcd}
\end{equation}

By the construction in Proposition \ref{prop:decomposition}, $a$ is the element in $\mathrm{H}^2(k,\mu_n)$ obtained by taking $\sigma_{\xi*} \pi_\xi^*(x) \in \mathrm{H}^2_\et(Y_\xi, \mu_n)$, restricting to $\mathrm{H}^2_\et(T, \mu_n)$, corestricting to $\mathrm{H}^2(k, \mu_n)$ and then multiplying by $m$. In other words, it is enough to show that $\tau_3 \sigma_{\xi*} \pi^*_\xi(x)$ does not depend on $\xi$.

Let $y = \sigma^*_\xi\sigma_{\xi*} \pi^*_\xi(x) - \pi^*_\xi(x)$ in $\mathrm{H}^2_\et(Y''_\xi, \mu_n)$. We claim that $\rho^* \tau_2(y) = 0$ in $\mathrm{H}^2_\et(T, \mu_n)$. This will follow from the facts that $T = \oplus_i k_i$, and for any closed point $\iota : \mathrm{Spec}(K) \hookrightarrow Y''_\xi$, we have that $\iota^*(y) = 0 \in \mathrm{H}^2(K, \mu_n)$. To show the latter fact, we note that $\sigma_{\xi*}(y) = 0$ since $\sigma_{\xi*}\sigma_\xi^* = \mathrm{id}$, which implies that restricting $y$ to $\mathrm{H}^2_\et(Y_{\xi 0}, \mu_n)$ is also $0$. Then, we have that

\begin{equation*}
\begin{tikzcd}
{\mathrm{H}^2_\et(Y_\xi'', \mu_n)} \arrow[r] \arrow[d] & \mathrm{Br}(Y_\xi'') \arrow[d, hook] \\
{\mathrm{H}^2_\et(Y_{\xi 0}, \mu_n)} \arrow[r]          & \mathrm{Br}(Y_{\xi 0})               
\end{tikzcd}
\end{equation*}

The injection $\mathrm{Br}(Y_\xi'') \hookrightarrow \mathrm{Br}(Y_{\xi 0})$ is due to Grothendieck's Purity theorem for the Brauer group since each irreducible component of $D''$ is simply connected. This implies that the image of $y$ in $\mathrm{Br}(Y_{\xi 0})$ is $0$ so the image of $y$ in $\mathrm{Br}(Y''_\xi)$ must also be $0$ due to the injectivity of the right hand side map. Finally, we use the fact that $\mathrm{H}^2(K, \mu_n) \to \mathrm{Br}(K)$ is injective coming from the Kummer sequence and Hilbert's Theorem 90.

Now, since $\rho^* \tau_2(y) = 0$, the definition of $y$ tells us that $$\rho^*\tau_2\pi_\xi^*(x) = \rho^* \tau_2 \sigma^*_\xi \sigma_{\xi*} \pi^*_\xi(x).$$ By the commutativity of the right square of \ref{eq:diagram}, and the fact that $\rho^*$ is a section tell us that $\rho^*\tau_2 \sigma^*_\xi = \tau_3$. Therefore, we have $$\rho^* \tau_2 \pi^*_\xi(x) = \tau_3 \sigma_{\xi*} \pi^*_\xi(x).$$ Then the commutativity of the left square gives us that $\tau_2 \pi^*_\xi(x) = \tau_1(x)\pi^*|_{D''}$. Therefore, $$\rho^* \tau_2 \pi_\xi^*(x) = \rho^* \tau_1 \pi^*|_{D''}(x),$$ and this last quantity, and therefore $a$, does not depend on $\xi$.
\end{proof}

\begin{theorem}
For any subgroup $B \subseteq \mathrm{Br}(X)$, if $X(\mathbb{A}_k)^B \neq \emptyset$, then $X(\mathbb{A}_k)^{B + \mathrm{Br}(X)[p^\perp]} \neq \emptyset$
\end{theorem}
\begin{proof}
By Corollary \ref{cor:finiteBrauer}, $\mathrm{Br}(X)/\mathrm{Br}_0(X)$ is finite. Therefore, we can assume that $B$ is finite, since we can replace $B$ with a finite subgroup that has the same image in $\mathrm{Br}(X)/\mathrm{Br}_0(X)$. Furthermore, we can replace $B$ with $B[p^\infty]$, where $B[p^\infty]$ is the $p$-primary subgroup of $B$, since $X(\mathbb{A}_k)^{B[p^\infty]}$ will also be nonempty and $B + \mathrm{Br}(X)[p^\perp] = B[p^\infty] + \mathrm{Br}(X)[p^\perp]$. Also, again because $\mathrm{Br}(X)/\mathrm{Br}_0(X)$ is finite, $$\mathrm{Br}(X)[p^\perp]/(\mathrm{Br}(X)[p^\perp] \cap \mathrm{Br}_0(X))$$ is finitely generated. By taking the least common multiple of the orders of these finite generators, we get an integer $n$ coprime to $p$ such that the images of $\mathrm{Br}(X)[p^\perp]$ and $\mathrm{Br}(X)[n]$ in $\mathrm{Br}(X)/\mathrm{Br}_0(X)$ are the same. Furthermore, by taking a suitable power of $n$, we can assume that $n \equiv 1 \pmod{p}$. Therefore, we look to show that $X(\mathbb{A}_k)^{B + \mathrm{Br}(X)[n]} \neq \emptyset$.

We use the maps $$\mathrm{H}^2_\et(X, \mu_n) \xrightarrow{\pi^*} \mathrm{H}^2_\et(Y', \mu_n) \xrightarrow{\sigma_*} \mathrm{H}^2_\et(Y, \mu_n)$$ where the second map is the composition of the maps $\mathrm{H}^2_\et(Y', \mu_n) \to \mathrm{H}^2_\et(Y_0, \mu_n) \to \mathrm{H}^2_\et(Y, \mu_n)$. Here, the last map is the inverse of the isomorphism $\mathrm{H}^2_\et(Y, \mu_n) \to \mathrm{H}^2_\et(Y_0, \mu_n)$. This, combined with the Kummer sequence, gives us a diagram

\begin{equation}
\label{eq:Kummer}
\begin{tikzcd}
{\mathrm{H}^2_\et(Y, \mu_n)} \arrow[r] & {\mathrm{Br}(Y)[n]}  \arrow[r] & 0 \\
{\mathrm{H}^2_\et(X,\mu_n)} \arrow[r]  \arrow[u, "\sigma_*\pi^*"]          & {\mathrm{Br}(X)[n]} \arrow[u, "\sigma_*\pi^*"] \arrow[r]           & 0
\end{tikzcd}
\end{equation} Furthermore, we get a similar diagram replacing $Y$ with any twist $Y_\xi$ and the maps with the respective maps $\sigma_{\xi*}$ and $\pi^*_\xi$.

Now, if we start with some $\mathcal{B} \in \mathrm{Br}(X)[n]$, we can find some lift $x \in \mathrm{H}^2_\et(X, \mu_n)$. Diagram \ref{eq:Kummer} tells us that $\sigma_{\xi*}\pi^*_\xi(\mathcal{B}) \in \mathrm{Br}(Y_\xi)[n]$ is mapped to by $\sigma_{\xi*}\pi^*_\xi(x) \in \mathrm{H}^2_\et(Y_\xi, \mu_n)$. Therefore, by Corollary \ref{cor:image}, there is an $a \in \mathrm{H}^2(k, \mu_n)$ such that $\sigma_{\xi*}\pi^*_\xi(x) = a \in \mathrm{H}^2_\et(Y_\xi, \mu_n)$.

Now, if we have $(P_v) \in X(\mathbb{A}_k)^B$, then we can look at the preimage of this point under $\pi$ in $Y'$. For each $v$, $P_v \in X(k_v)$ might not lift to a rational point on $Y'(k_v)$. However, the preimage of $P_v$ must either be a single point or $p$ distinct (geometric) points since $X$ is the quotient of $Y'$ by the action of $\zeta$. If the preimage is a single point, then it must be rational, and if the preimage is $p$ points, then they are either all rational or live in a cyclic degree $p$ extension $F_v$ of $k_v$. This implies that, for each place $v$ of $k$, we can take a twist $\xi_v$ corresponding to the extension $F_v$ such that the twist $Y'_{\xi_v}$ of $Y'_{k_v}$ has $P_v$ lifting to rational points on $Y'_{\xi_v}(k_v)$. Furthermore, we can note that there is a extension $F$ of $k$ such that the completion of $F$ at $v$ is exactly $F_v$ by the Grunwald-Wang theorem, e.g. \cite[Theorem 9.2.8]{CohomologyNumberFields}. This implies that $Y'_{\xi_v} = Y'_{\xi} \times_k k_v$ which means that $Y'_{\xi_v}$ is defined over $k$.

All of this means that, for each place $v$ of $k$, we can find a point $R_v' \in Y'_{\xi_v}(k_v)$ such that $\pi_{\xi_v}(R_v') = P_v$. Then, we let $R_v = \sigma_{\xi_v}(R_v')$. Furthermore, since we assume $B$ to be finite, $X(\mathbb{A}_k)^B$ is open in $X(\mathbb{A}_k)$ and so we can perturb $(P_v)$ slightly so that $R_v$ avoids the fixed locus of $\zeta_{\xi_v}$ in $Y_{\xi_v}$. This implies that $R_v'$ is the unique point in $Y'_{\xi_v}$ such that $R_v = \sigma_{\xi_v}(R_v')$. Then, we have that $\mathcal{B}(P_v) = \sigma_{\xi_v*}\pi_{\xi_v}^*(\mathcal{B})(R_v)$. As we have seen, $\sigma_{\xi_v *}\pi_{\xi_v}^*(\mathcal{B})$ is the image of $a \in \mathrm{H}^2(k, \mu_n)$, and since this is a global Brauer class and independent of $F$, the sum of the local Brauer classes $\sigma_{\xi_v *}\pi_{\xi_v}^*(\mathcal{B})(M_v)$ must be zero. This implies that the sum of the $\mathcal{B}(Q_v)$ is also zero. Therefore, since $\mathcal{B}$ was an arbitrary element of $\mathrm{Br}(X)[n]$, this implies that $(Q_v)$ lies in $X(\mathbb{A}_k)^{\mathrm{Br}(X)[n]} = X(\mathbb{A}_k)^{\mathrm{Br}(X)[p^\perp]}$, and since $(Q_v)$ is also in $X(\mathbb{A}_k)^B$, we have that $(Q_v) \in X(\mathbb{A}_k)^{B + \mathrm{Br}(X)[p^\perp]}$.
\end{proof}

\begin{remark}
If we let $B$ be the trivial subgroup of $\mathrm{Br}(X)$ in the previous theorem, then the statement becomes $$X(\mathbb{A}_k) \neq \emptyset \Rightarrow X(\mathbb{A}_k)^{\mathrm{Br}(X)[p^\perp]} \neq \emptyset.$$
\end{remark}

\section{Application}
\begin{proposition}
Let $C$ be the curve given by $ax^3 + by^3 + cz^3 = 0$ in $\mathbb{P}^2_\mathbb{Q}$, then $C$ is a torsor of the elliptic curve $E := y^2 = x^3 - 144a^2b^2c^2$. If we look at the surface $C \times C$, it is a torsor of $E \times E$, and we have an automorphism $\rho$ on $C \times C$ given by $(P,Q) \mapsto (Q,R)$, where the points $P,Q,R$ are colinear. This automorphism has order $3$ and so gives rise to a generalized Kummer variety $X$ using the resolution in Example \ref{ex:E}. Then, there is no Brauer-Manin obstruction to the local-global principle on $X$.
\end{proposition}
\begin{proof}
This follows from the fact in \cite{van2007cubic} that $$\mathrm{Br}(X)/\mathrm{Br}_0(X) \cong  \begin{cases} \mathbb{Z}/2\mathbb{Z} &\text{if } 4abc \text{ is a cube in } \mathbb{Q} \\ 0 &\text{otherwise}
\end{cases}$$ and the fact that this automorphism is order $3$ so that the $3$-primary subgroup of $\mathrm{Br}(X)/\mathrm{Br}_0(X)$ is always trivial.
\end{proof}
\begin{remark}
This was also proven by other means in \cite{navone2025transcendentalbrauergroupscubic}.
\end{remark}
\begin{remark}
If we assume Skorobogatov's conjecture that the Brauer-Manin obstruction to K3 surfaces is the only one, then the family of varieties satisfies the local-global principle.
\end{remark}
\section*{Acknowledgements}
First and foremost, I want to thank my advisor Joe Silverman for all the great geometric intuition he has imparted on me and for his wide breadth of knowledge. I want to thank Alexei Skorobogatov for help answering my questions. I would like to thank Giorgio Navone for many helpful discussions and insights.
\printbibliography
\end{document}